\def\marker{\>\hbox{${\vcenter{\vbox{
    \hrule height 0.4pt\hbox{\vrule width 0.4pt height 6pt
    \kern6pt\vrule width 0.4pt}\hrule height 0.4pt}}}$}\>}

\documentclass[12pt]{article}

\usepackage[left=2cm,right=2cm,top=3cm,bottom=3cm]{geometry}

\usepackage{amsmath, amssymb, latexsym, amsthm}
\usepackage{fullpage}
\usepackage[flushmargin]{footmisc}

\newtheorem{theorem}{Theorem} 
\newtheorem{theorem*}{Theorem} 

\newtheorem{lemma}[theorem]{Lemma}

\theoremstyle{definition}

\newtheorem{question}{Question}

\newtheorem{construction}{Construction}

\theoremstyle{remark}
\newtheorem{claim}{Claim}

{\unskip\nobreak\hskip 1em plus 1fil\nobreak$\Box$
\parfillskip=0pt%
\endtrivlist}

{\end{enumerate}}

\usepackage[dvips]{graphicx}

\graphicspath{{counterfigs/}}

\newcommand{\ex}{\text{ex}}

\DeclareMathOperator{\sat}{sat}

\title{Graph Saturation in Multipartite Graphs}
\author{Michael Ferrara\footnotemark[1] \footnotemark[3], Michael S.\ Jacobson\footnotemark[1],\\ Florian Pfender\footnotemark[1], and Paul S.\ Wenger\footnotemark[2]}

\begin{document}

\maketitle

\begin{abstract}

Let $G$ be a fixed graph and let ${\mathcal F}$ be a family of graphs.  A subgraph $J$ of $G$ is \textit{${\mathcal F}$-saturated} if no member of ${\mathcal F}$ is a subgraph of $J$, but for any edge $e$ in $E(G)-E(J)$, some element of ${\mathcal F}$ is a subgraph of $J+e$.  We let $\ex({\mathcal F},G)$ and $\sat({\mathcal F},G)$ denote the maximum and minimum size of an ${\mathcal F}$-saturated subgraph of $G$, respectively.  If no element of ${\mathcal F}$ is a subgraph of $G$, then $\sat({\mathcal F},G) = \ex({\mathcal F}, G) = |E(G)|$.

In this paper, for $k\ge 3$ and $n\ge 100$ we determine $\sat(K_3,K_k^n)$, where $K_k^n$ is the complete balanced $k$-partite graph with partite sets of size $n$.  We also give several families of constructions of $K_t$-saturated subgraphs of $K_k^n$ for $t\ge 4$.  Our results and constructions provide an informative contrast to recent results on the edge-density version of $\ex(K_t,K_k^n)$ from [A. Bondy, J. Shen, S. Thomass\'e, and C. Thomassen, Density conditions for triangles in multipartite graphs, {\it Combinatorica} \textbf{26} (2006), 121--131] and [F. Pfender, Complete subgraphs in multipartite graphs, {\it Combinatorica} \textbf{32} (2012), no. 4,  483--495].     

\end{abstract}

\renewcommand{\thefootnote}{\fnsymbol{footnote}}
\footnotetext[1]{Dept.\ of Mathematical and Statistical Sciences, Univ.\ of Colorado Denver, Denver, CO; email addresses:
{\tt michael.ferrara@ucdenver.edu}, {\tt michael.jacobson@ucdenver.edu},
{\tt florian.pfender@ucdenver.edu}.}
\footnotetext[2]{
School of Mathematical Sciences, Rochester Inst.\ of Technology, Rochester, NY;
{\tt pswsma@rit.edu}.}
\footnotetext[3]{Research supported in part by Simons Foundation Collaboration Grant \# 206692.}
\renewcommand{\thefootnote}{\arabic{footnote}}

\baselineskip18pt

\section{Introduction}

All graphs in this paper are simple.  Let $N(v)$ and $N[v]$ denote the open and closed neighborhoods of a vertex $v$, respectively, and for a set of vertices $S$, let $N(S) = \bigcup_{x\in S} N(S)$.  The set $N[S]$ is defined similarly.  Further, $d(v)$ denotes the degree of a vertex $v$, and $\delta(G)$ denotes the minimum degree of a graph $G$. Given two sets of vertices $X$, and $Y$, we let $E(X,Y)$ denote the set of edges joining $X$ and $Y$.  Central to this paper is $K_k^n$, the complete balanced $k$-partite graph with partite sets of size $n$.  Throughout, $V_1, V_2, \ldots,V_k$ will be the partite sets of $K_k^n$ such that $V_i=\{v_i^1,v_i^2,\ldots ,v_i^n\}$ for each $i\in\{1,\ldots,k\}$.  Furthermore, to avoid certain degeneracies, we assume that $k\ge 3$ and that $n\ge 2$.

Given a family of graphs ${\mathcal F}$, a graph $G$ is {\it ${\mathcal F}$-saturated} if no element of ${\mathcal F}$ is a subgraph of $G$, but for any edge $e$ in the complement of $G$, some element of ${\mathcal F}$ is a subgraph of $G+e$.  If ${\mathcal F}=\{H\}$, then we say that $G$ is {\it $H$-saturated}.  The classical extremal function $\ex(H,n)$ is the maximum number of edges in an $n$-vertex $H$-saturated graph.  Erd\H{o}s, Hajnal and Moon~\cite{EHM} studied $\sat(H,n)$, the {\it minimum} number of edges in an $n$-vertex $H$-saturated graph, and determined $\sat(K_t,n)$.  The value of $\sat(H,n)$ is known precisely for very few choices of $H$, and the best upper bound on $\sat(H,n)$ for general $H$ appears in \cite{KT}. It remains an interesting problem to determine a non-trivial lower bound on $\sat(H,n)$.  For a thorough survey of results on the $\sat$ function, we refer the reader to~\cite{SatSurvey}.  

The focus of this paper is the study of ${\mathcal F}$-saturated subgraphs of a general graph.  Specifically, let $G$ be a fixed graph and let ${\mathcal F}$ be a family of graphs.  A subgraph $J$ of $G$ is \textit{${\mathcal F}$-saturated} if no member of ${\mathcal F}$ is a subgraph of $J$, but for any edge $e$ in $E(G)-E(J)$, some element of ${\mathcal F}$ is a subgraph of $J+e$.  We let $\ex({\mathcal F},G)$ and $\sat({\mathcal F},G)$ denote the maximum and minimum size of an ${\mathcal F}$-saturated subgraph of $G$, respectively.  If no element of ${\mathcal F}$ is a subgraph of $G$, then $\sat({\mathcal F},G) = \ex({\mathcal F}, G) = |E(G)|$.  Note as well that $\sat(H,n) = \sat(H, K_n)$ and $\ex(H,n) = \ex(H,K_n)$.        

The problem of determining $\sat({\mathcal F},G)$ for general $G$ was first proposed in \cite{EHM} and Erd\H{o}s notably studied $\ex(K_3,G)$ (amongst other related problems) in \cite{Erd}.  Subsequently Bollob\'{a}s \cite{BolBip1, BolBip2} and Wessel \cite{WesBip1,WesBip2} independently determined $\sat(K_{a,b},K_{m,n})$ as a corollary to results on a related, but more specific problem.  These results were extended to the setting of $k$-partite, $k$-uniform hypergraphs by Alon \cite{AlonMultHyp} and were also generalized by Pikhurko in his Ph.D. Thesis \cite{PikThesis}.  Additionally several bounds and exact results for $\sat(P_k,K_{m,n})$ and $\sat(Q_2,Q_k)$, where $Q_k$ denotes the $k$-dimensional hypercube, were given in \cite{DudBipPath} and \cite{CubeSat}, respectively.   The structure of ${\mathcal F}$-saturated subgraphs of a general graph were also examined via a combinatorial game in \cite{FHJGame}.

In this paper we study $\sat(K_t,K_k^n)$.  We determine $\sat(K_3,K_k^n)$ for $k\ge 4$ when $n$ is large enough, and $\sat(K_3,K_3^n)$ for all values of $n$.  For $t\ge 4$, we also provide constructions of $K_t$-saturated subgraphs of $K_k^n$ with few edges.  

The corresponding problem of determining $\ex(K_3,K_k^n)$ has received considerable attention recently.  When determining the maximum size of an $H$-free subgraph of a complete multipartite graph, frequently one studies the minimal number of edges joining any two partite sets rather than the total number of edges in the subgraph.  Consequently, results on the maximum size of $H$-free subgraphs of multipartite graphs are expressed in terms of edge-densities.  In 2006, Bondy, Shen, Thomass\'e, and Thomassen~\cite{BSTT} determined the maximum edge-density of triangle-free subgraphs of complete tripartite graphs.  Furthermore, they gave bounds on the edge density that guarantees that a subgraph of an infinite-partite graph with finite parts contains a triangle.  Pfender~\cite{Pfender} extended these results, determining the maximum density of a $K_k$-free subgraph of an $\ell$-partite graph for large enough $\ell$.  In contrast to the results on the extremal function in multipartite graphs, our results for $K_3$-saturated subgraphs of $K_k^n$ cannot be meaningfully expressed in terms of edge densities, as we demonstrate that the minimum saturated graphs often have edge density tending to zero within certain pairs of partite sets.

\section{$K_3$-saturated subgraphs of $K_k^n$}

In this section we examine $K_3$-saturated subgraphs of $K_k^n$.
In particular, for $n$ large enough we determine $\sat(K_3,K_k^n)$ for all $k$, and we determine $\sat(K_3,K_3^n)$ for all values of $n$.
First we provide two constructions for $K_3$-saturated subgraphs of $K_k^n$, either of which can be optimal depending on the relative sizes of $k$ and $n$.

\begin{construction}\label{exk>n}
Begin with the complete bipartite graph joining $V_1$ and $V_2$ and remove the edge $v_1^1v_1^2$.
Then join each vertex in $V_3\cup\ldots\cup V_k$ to both $v_1^1$ and $v_2^1$ (see Figure~\ref{K3pics}).
We call this graph $G_1$.
Thus,
 \[
  E(G_1)=\{ v_m^1v_i^j: 1\le m \le 2, 3\le i\le k, 1\le j\le n\}\cup \{ v_1^iv_2^j:i+j\ge 3\}.
 \]
\end{construction}

\begin{construction}\label{exn>k}
First join $V_1\setminus \{v_1^1\}$ to $v_2^1$ and $v_3^1$, join $V_2\setminus \{v_2^1\}$ to $v_1^1$ and $v_3^1$, and join $V_3\setminus \{v_3^1\}$ to $v_1^1$ and $v_2^1$.
Then join each vertex in $V_4\cup\ldots\cup V_k$ to $v_1^1$, $v_2^1$, and $v_3^1$ (see Figure~\ref{K3pics}).
We call this graph $G_2$.
Thus,
 \begin{align*}
  E(G_2)&=\{ v_m^1v_i^j: 1\le m\le 3, 4\le i\le k, 1\le j\le n\}\\
  &\qquad \cup \left\{ v_m^1v_i^j:(m,i)\in\{(1,2),(1,3),(2,3)\},j\ge 2\right\}.
 \end{align*}
\end{construction}

\begin{figure}\label{K3pics}
\begin{center}
\includegraphics[scale=.4]{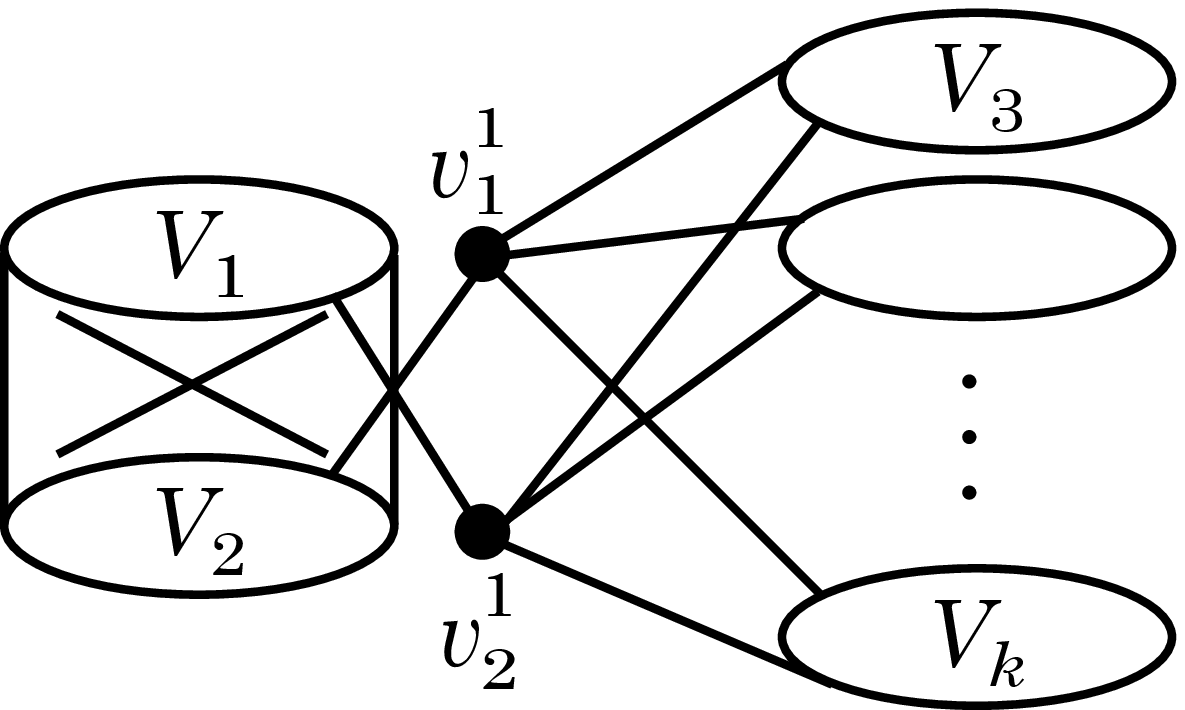}\hspace{1in}\includegraphics[scale=.4]{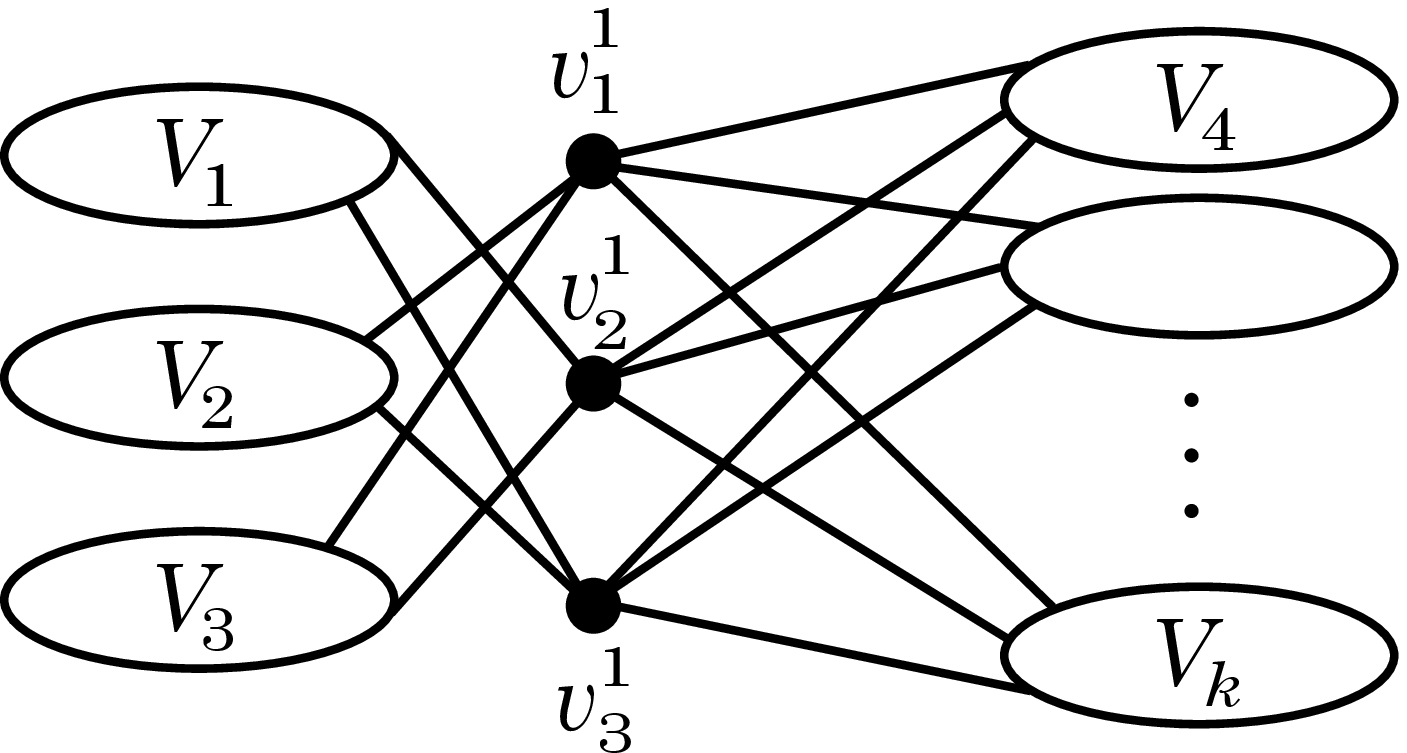}
\caption{Constructions~\ref{exk>n} (left) and \ref{exn>k} (right).}
\end{center}
\end{figure}

\begin{lemma}\label{K3ub}
The graphs in Constructions~\ref{exk>n} and~\ref{exn>k} are $K_3$-saturated subgraphs of $K_k^n$, and thus
\[\sat(K_3,K_k^n)\le		\begin{cases}
											2kn+n^2-4n-1			& \text{ if }k\ge n-1+5/n\\
											3kn-3n-6 				& \text{ if }k<n-1+5/n.
											\end{cases}
\]
\end{lemma}

\begin{proof}
The graphs in Constructions~\ref{exk>n} and~\ref{exn>k} are clearly $K_3$-saturated subgraphs of $K_k^n$, and 
\[|E(G_1)|=2kn+n^2-4n-1\]
and 
\[|E(G_2)|=3kn-3n-6.\]
Furthermore, $|E(G_1)|\le |E(G_2)|$ in the range given.
\end{proof}

We now determine $\sat(K_3,K_k^n)$ when $k\ge 4$ and $n$ is sufficiently large.

\begin{theorem}\label{K3main}
If $k\ge 3$ and $n\ge 100$, then
\[\sat(K_3,K_k^n)=\min\{2kn+n^2-4k-1, 3kn-3n-6\}.\]
Further, equality is only attained by the graphs in Constructions~\ref{exk>n} and~\ref{exn>k}, respectively.
\end{theorem}

To prove Theorem~\ref{K3main}, we consider two cases, depending on the minimum degree of a $K_3$-saturated subgraph of $K_k^n$ with the minimum number of edges. Each of the cases is treated in a separate lemma.

\begin{lemma}
If $k\ge 3$, $n\ge 100$,   
and $G$ is a $K_3$-saturated subgraph of $K^n_k$ with minimum degree $\delta(G)\ge 3$, 
then $|E(G)|\ge 3kn-3n$.
\end{lemma}
\begin{proof}
We proceed by contradiction, so suppose that $|E(G)|< 3kn-3n$. Clearly, $G$ has minimum degree at most five, as otherwise $|E(G)|\ge 3kn$.
\begin{claim}\label{cl1}
$G$ does not contain four independent vertices of degree at most $5$ with pairwise disjoint neighborhoods. 
\end{claim}
Suppose that $u_1,u_2,u_3$ and $u_4$ are independent vertices with pairwise disjoint neighborhoods. Since $G$ is saturated and the addition of the edge $u_iu_j$ cannot create a triangle in $G$, it must be that $u_1,u_2,u_3,u_4\in V_i$ for some $i$.  Furthermore every vertex $y\in V(G)\setminus V_{i}$ 
has a neighbor in $N[u_j]$ for $1\le j\le 4$. Thus,  
\[
|E(G)|\ge 4kn-4n-\tfrac32(d(u)+d(v)+d(w)+d(x))\ge 4kn-4n-30, 
\]
where the last term addresses the double counting of edges between the disjoint neighborhoods of $u_1$, $u_2$, $u_3$, and $u_4$.
For $n\ge 15$, this is a contradiction.

Throughout the remaining claims, let $\tilde{S}$ be a maximal set of vertices with the following properties:
\begin{enumerate}
\item $\tilde{S}$ is an independent set,
\item $\tilde{S}$ contains no vertex of degree $6$ or larger, and
\item For every $u\in \tilde{S}$, we have $|N(u)\cap N(\tilde{S}-u)|\le 5-d(u)$.\label{prop}
\end{enumerate}
A set $\tilde{S}$ with the above properties can easily be found by a greedy search as follows. First, greedily find a maximal independent set $\tilde{S}$ of vertices of degree $3$ without respect for property~\eqref{prop}. Then, add vertices of degree $4$ from $V\setminus N[\tilde{S}]$ to $\tilde{S}$ with property~\eqref{prop} one-by-one.
If the addition of such a vertex $u$ prompts another previously added vertex $v$ of degree $d(v)=4$ to lose property~\eqref{prop}, then $v$ is the only vertex in $\tilde{S}$ with $N(u)\cap N(v)\ne\emptyset$.
Delete that vertex $v$ from $\tilde{S}$, with the consequence that then $N(u)\cap N(\tilde{S}-u)=\emptyset$. Note that with each additional vertex, $|\tilde{S}|$ grows by one, or $|\tilde{S}|$ stays the same and the number of vertices $u$ with $N(u)\cap N(\tilde{S}-u)=\emptyset$ increases by one. As there can be at most three such vertices by Claim~\ref{cl1}, $|\tilde{S}|$ has to grow by one at least once for every four steps. Thus, this process will terminate, and all vertices $u$ of degree $4$ that remain in $V\setminus N[\tilde{S}]$ have at least two neighbors in $N(\tilde{S})$. Now add vertices $u\in V\setminus N[\tilde{S}]$ of degree $5$ with $N(u)\cap N(\tilde{S})=\emptyset$ one-by-one. This does not affect property~\eqref{prop} for any of the previous vertices. Finally, remove all vertices $u$ of degree $3$ from $\tilde{S}$ one-by-one for which $N(u)\subseteq N(\tilde{S}-u)$. Note that this does not change $N(\tilde{S})$, as each of these removed vertices has all three of its neighbors in $N(\tilde{S})$.
Let $\tilde{X}=N(\tilde{S})$, and $\tilde{L}=V\setminus (\tilde{S}\cup\tilde{X})$.   
\begin{claim}\label{cl2}
$|E(\tilde{X})|< 3(|\tilde{X}|-n)$, so that in particular $|\tilde{X}|\ge n+1$.
\end{claim}
Assume otherwise, and observe that the conditions on $\tilde{S}$ imply that for every $v\in\tilde{L}$, $d_{\tilde{X}}(v)+\tfrac12d_{\tilde{L}}(v)\ge 3$.  We therefore have that
\begin{align*}
|E(G)|\ge&~ |E(\tilde{X})|+|E(S,\tilde{X})|+|E(\tilde{L},\tilde{X})|+|E(\tilde{L})|\\
\ge&~ 3|\tilde{X}|-3n+3|\tilde{S}|+\sum_{v\in \tilde{L}}(d_{\tilde{X}}(v)+\tfrac12d_{\tilde{L}}(v))\\
\ge&~ 3|\tilde{X}|-3n+3|\tilde{S}|+3|\tilde{L}|\\
=&~3kn-3n,
\end{align*}
a contradiction showing the claim.   

Now, let $S\subseteq \tilde{S}$ be the set of vertices $s\in \tilde{S}$ with $N(s)\cap N(\tilde{S}-s)\ne \emptyset$ and let $X=N(S)$.
\begin{claim}\label{cl2a}
$|E(X)|< 3(|X|-n)$.
\end{claim}
Let $S'$ be the set of vertices $s$ satisfying $N(s)\cap N(\tilde{S}-s)= \emptyset$, and let $|S'|=m$.
If $S'=\emptyset$, then the claim follows immediately from~\ref{cl2}.
If $S'\neq \emptyset$, then there is a vertex $s\in S'$ such that $N(s)\cap X=\emptyset$.
Thus all vertices in $\tilde{S}$ must be in the same partite set $V_{i}$, and all vertices
in $N(\tilde{S})$ must be in other partite sets.
Furthermore, there is a path of length $2$ joining each vertex in $X$ to each vertex in $S'$, so $|E(X,N(S')|\ge m|X|$.
If $|E(X)|\ge 3(|X|-n)$, it follows that 
\begin{align*}
|E(\tilde{X})|&\ge 3(|X|-n)+m|X|\\
&\ge 3(|\tilde{X}|-5m-n)+m(|\tilde{X}|-5m)\\
& = 3(|\tilde{X}|-n)+m(|\tilde{X}|-5m-15).
\end{align*}
By Claim~\ref{cl1}, $m\le 3$, so this is a contradiction.
%

\begin{claim}\label{clZ}
There exists a set $Z\subset X$ such that $|Z|\le 4$ and $S\subseteq N(Z)$.
\end{claim}
Let $Z\subset X$ be minimum with $S\subset N(Z)$ and suppose first that $S\subset V_{i}$ for some $i$. By the minimality of $Z$, for each $x\in Z$ there is some $s\in S$ such that $N_Z(s)=\{z\}$.  Hence if $|Z|\ge 5$, then every vertex not in $V_{i}$ is adjacent to at least $5$ vertices in $V_{i}\cup N(V_{i})$, and every vertex in $N(V_{i})$ is adjacent to at least one vertex in $V_{i}$. Being careful not to double count edges within $N(V_{i})$, we get
\[
|E(G)|\ge (1+\tfrac42)|N(V_{i})|+5(kn-n-|N(V_{i})|)\ge 3kn-3n,
\]
a contradiction.
Otherwise suppose that for distinct $i$ and $j$ there are vertices $s_i\in V_i$ and $s_j\in V_j$ in $S$.
By property~\eqref{prop}, every vertex in $S\setminus V_i$ is adjacent to one of at most two neighbors of $s_i$.
Similarly, every vertex in $S\setminus V_j$ is adjacent to one of at most two neighbors of $s_j$.
Thus there is a set of at most four vertices in $X$ whose combined neighborhood contains $S$.
%

Let $Z=\{z_{1},z_{2},\ldots,z_{|Z|}\}$.
Let $Y=\{y\in X\setminus Z:|N(y)\cap S|=1\}$, and $W=X\setminus Y$. For $y\in Y$, let $s_{y}\in S$ be the unique vertex in $S$ with $y\in N(s)$. 
Let 
\begin{align*}
S_{i}&=S\cap N(z_{i})\setminus N(\{z_{1},\ldots ,z_{i-1}\}),\mbox{ and}\\
Y_{i}&=Y\cap N(S_{i}).
\end{align*}
For $w\in W\setminus \{z_{1},\ldots,z_{i}\}$, let
\[
Y_{i}(w)=\{y\in Y:\{z_{i},w\}\subset N(s_{y})\}.
\] 
Suppose that $y$ and $y'$ are distinct vertices in $Y_i(w)$ and note that since $s_y$ and $s_{y'}$ share two neighbors in $X$, the conditions imposed on $\tilde{S}$ imply that $d(s_y)=d(s_{y'})=3$.  Consequently, either $yy'\in E(X)$, or both $s_{y'},y\in V_{\ell}$ and $s_{y},y'\in V_{j}$ for some $j$ and $\ell$. Otherwise, we would have $N(s_{y})\cap N(y')=N(s_{y'})\cap N(y)=\emptyset$, a contradiction to the assumption that $G$ is $K_3$-saturated.  Note this implies that we can never have both $y,y'\in V_{j}$ as $V_{j}$ is an independent set. Therefore each vertex in $Y_i(w)$ is adjacent to all but at most one other vertex in $Y_i(w)$, so we have
\[
|E(Y_{i}(w))|\ge\tfrac12|Y_{i}(w)|(|Y_{i}(w)|-2).
\]

Partition $S_i$ into sets $S_i^{(1)},\ldots,S_i^{(d_i)}$ such that $s$ and $s'$ are in the same set if and only if the have a common neighbor in $W\setminus\{z_1,\ldots,z_i\}$.
For each $S_i^{(j)}$ pick a vertex $s_i^{(j)}\in S_i^{(j)}$ and a vertex $y_i^{(j)}\in N(s_i^{(j)})\cap Y$.
Finally, assign each pair $(y_i^{(j)},s_i^{(j)})$ the label $(p,q)$ where $y_i^{(j)}\in V_{p}$ and $s_i^{(j)}\in V_q$.
Given two such pairs $(y_i^{(j)},s_i^{(j)})$ and $(y_i^{(\ell)},s_i^{(\ell)})$ with labels $(p,q)$ and $(p', q')$, respectively, there is an edge joining $y_i^{(j)}$ to $N(s_i^{(\ell)})$ whenever $p\neq q'$.
Thus, when we consider $(y_i^{(j)},s_i^{(j)})$ and $(y_i^{(\ell)},s_i^{(\ell)})$ we count
\begin{center}
\begin{tabular}{ll}
$0$ edges & if $p=q'$ and $q=p'$;\\
$1$ edge & if ($p=q'$ and $q\neq p'$) or ($p\neq q'$ and $q=p'$);\\
$1$ edge & if $p\neq q'$, $q\neq p'$, and $p\neq p'$;\\
$2$ edges & if $p=p'$.
\end{tabular}
\end{center}
Given $p,q\in \{1,\ldots,d_i\}$, let $X_{p,q}$ denote the number of pairs with label $(p,q)$.
Thus there are at least 
\begin{align*}\label{Yedges}
\binom{d_i}{2}+\sum_{(p,q)}\binom{X_{p,q}}{2} -\sum_{p<q}X_{p,q}X_{q,p} &= \binom{d_i}{2}-\frac{1}{2}d_i+\frac{1}{2}\sum_{p<q}(X_{p,q}-X_{q,p})^2\\
& \ge \frac{1}{2}d_i(d_i-2)
\end{align*}
edges incident to $\{y_i^{(1)},\ldots,y_i^{(d_i)}\}$ that do not have both endpoints in $Y_i(w)$ for some $w$.
Consequently there are at least 
\[
\tfrac12d_{i}(d_{i}-2)+\sum_{w}\tfrac12|Y_{i}(w)|(|Y_{i}(w)|-2)
\]
edges incident to $Y_i$, none of which have endpoints in $Y_j$ for $j\neq i$.

Summing up over all $z_{i}$, we get
\[
|E(X)|\ge \sum_{i=1}^{|Z|}\left(\tfrac12d_{i}(d_{i}-2)+\sum_{w }\tfrac12|Y_{i}(w)|(|Y_{i}(w)|-2)\right).
\]
This bound is minimized for fixed $|X|$ when all of the $|N(S_{i})|$ and all of the $|Y_{i}(w)|$ are as equal as possible, and $|Z|$ is maximized, i.e.\ $|Z|=4$. Further, we may modify $W$ as follows so that there are no $s\in S_{i}$ with $N(s)\cap W=\{z_{i}\}$.
If such a vertex has neighborhood $N(s)=\{z_i,y,y'\}$, then add $y$ to $W$, so that $Y_i(y)=\{y'\}$.
If such a vertex has neighborhood $N(s)=\{z_i,y,y',y''\}$, then add $y$ to $W$, so that $Y_i(y)=\{y',y''\}$. In either case, note that $d_i$ is unchanged and we add a term of at most zero to the sum, so the bound will not increase.  Relaxing all integrality constraints and setting $|N(S_{i})|=\tfrac14|X|$ and $|Y_{i}(w)|=\frac1d(\tfrac14|X|-d-1)$, the bound only depends on $|X|$ and $d=d_{i}$ (note that $|X|=|Y|+4(d+1)$). We get
\[
|E(X)|\ge 2d(d-2)+\tfrac2{d}(\tfrac{|X|}{4}-d-1)^2-|X|+4d+4,
\]
and thus
\begin{align*}
|E(X)|-3|X|&\ge 2d(d-2)+\tfrac2{d}(\tfrac{|X|}{4}-d-1)^2-4|X|+4d+4\\
&=8 + 2 d +  2d^2 + \frac{(|X|-4 )^2}{8 d} - 5 |X|.
\end{align*}
Given $d>0$ and $|X|>0$, the right side is minimized for $d=12$ and $|X|=244$, and thus
\[
|E(X)|-3(|X|-n)\ge 3n-300\ge 0,
\]
a contradiction to Claim~\ref{cl2a}.
\end{proof}
Note that we are very generous with our bound on $|E(X)|$. We heavily undercount the edges between $Y_i(w)\cup w$ and $Y_i(w')\cup w'$, and we do not count the edges between $Y_i$ and $Y_j$ at all. Further note that for the case that $S\subseteq V_i$, the bound can easily be improved to
\[
|E(X)|\ge \sum_{i=1}^{|Z|}\left(d_{i}(d_{i}-1)+\sum_{w}\tfrac12|Y_{i}(w)|(|Y_{i}(w)|-1)\right).
\]
For the case that $S$ contains vertices in both $V_i$ and $V_j$, it is not hard to see that $|Z|\le 3$. This can be further lowered to $|Z|=1$ if one treats a few exceptional cases. All these arguments can be used to lower the bound on $n$ in the lemma, but the technicalities involved are too great to justify their exposition here, especially as as one would still need to require $n\ge 20$ or so.

As there cannot be a vertex of degree less than $2$ in a $K_3$-saturated graph, it only remains to consider the case where $\delta(G)=2$ in order to complete the proof of Theorem~\ref{K3main}.  

\begin{lemma}
If $n\ge 8$, $k\ge 3$, and $G$ is a $K_3$-saturated subgraph of $K^n_k$ of minimum size with minimum degree $2$, then $G$ is one of the graphs from Constructions 1 and 2, and in particular
\[
|E(G)|=\min\{ 2kn+n^2-4n-1,3kn-3n-6\}.
\]
\end{lemma}
\begin{proof}
Without loss of generality, assume that $N(v_3^1)=\{v_1^1,v_2^1\}$. Partition the vertices as follows:
\begin{align*}
A_i&= \{u\in V_i:v_1^1,v_2^1\notin N(u)\},\\
B_i&= \{u\in V_i:v_1^1\in N(u),v_2^1\notin N(u)\},\\
C_i&= \{u\in V_i:v_1^1\notin N(u),v_2^1\in N(u)\},\mbox{ and}\\
D_i&= \{u\in V_i:v_1^1,v_2^1\in N(u)\}.
\end{align*}
Note that $B_1=D_1=C_2=D_2=\emptyset$.  Also, $A_1=\{v_1^1\}$, $A_2=\{v_2^1\}$ and $A_\ell=\emptyset$ for $\ell\ge 4$, as $G$ is $K_3$-saturated and $N(v_i^j)\cap N(v_3^1)\ne \emptyset$ for $1\le i\le 2$ and $2\le j\le n$, and for $4\le i\le k$ and $1\le j\le n$.

Let $A=\bigcup A_i$, $B=\bigcup B_i$, $C=\bigcup C_i$, and $D=\bigcup D_i$. Note that $B\cup D$ and $C\cup D$ are independent sets, lest $G$ contain a triangle. Thus, in particular, $N(B)\subseteq (A\setminus \{v_2^1\})\cup C$ and $N(C)\subseteq (A\setminus\{v_1^1\})\cup B$.

First, consider the case that $A_3=\emptyset$. Then, for every $u\in D$, $N(u)=\{v_1^1,v_2^1\}$. 
Further, the sets $C_i$ and $B_j$ induce a complete bipartite graph for any $i\ne j$ as the intersection of their neighborhoods is empty. Thus, once given the sizes of the $B_i$ and $C_i$, $G$ is completely determined. Note that every vertex in $B\cup C$ has degree at least $n$, whereas vertices in $D$ have degree $2$. Thus, $|E(G)|$ is minimized if $|B_i|=|C_i|=0$ for $3\le i\le k$, which yields the graph in Construction 1.

Now suppose that $|A_3|=1$, say $A_3=\{v_3^2\}$. 
Further suppose that $u\in B\setminus B_2$. If $uv_3^2\notin E(G)$, then $C_1\subseteq N(u)$. If, on the other hand, $uv_3^2\in E(G)$, then $N(u)$ contains a vertex in $C\setminus N(v_2^2)$, as otherwise there is no path of length at most $2$ from $u$ to $v_2^1$. Analogous statements hold for vertices $w\in C\setminus C_1$. Further, $D\setminus D_3\subseteq N(v_3^2)$.
This implies that 
\begin{align*}
|E(G)|\ge& ~3|D\setminus D_3|+n|B\setminus (B_2\cup N(v_3^2))|+n|C\setminus (C_1\cup N(v_3^2))|+\tfrac{n}{2}|(C_1\cup B_2)\setminus N(v_3^2)|\\
&+3|B\cap N(v_3^2)\setminus  B_2|+3|C\cap N(v_3^2)\setminus  C_1|+2|((C_1\cup B_2)\cap N(v_3^2))\cup D_3|\\
 \ge& ~3|D\setminus D_3|+3|B\setminus  B_2|+3|C\setminus  C_1|+2|C_1\cup B_2\cup D_3|\\
=& 3(|V(G)|-|A|) - |C_1\cup B_2\cup D_3|\\
=& 3(kn-3) -(|C_1|+|B_2|+|D_3|)\\
  \ge&~3kn-3n-6.
\end{align*}
Note that equality holds only if $B\subset  N(v_3^2)$ and $C\subset N(v_3^2)$, which then implies that in fact $B=B_2$ and $C=C_1$.  It then follows that $G$ is the graph from Construction 2.

Finally suppose that $|A_3|\ge 2$. 
To count the edges, we assign a charge of $1$ to each edge $uw$ and distribute the charge onto $u$ and $w$ as follows in this order, taking symmetry into account:
\begin{align*}
u\in A && 0\to u, ~1\to w\\
|N(u)\cap A|\ge 3\mbox{ and }|N(w)\cap A|\le 2 && 0\to u, ~1\to w\\
\mbox{otherwise} && .5\to u, .5\to w
\end{align*}

If every vertex in $B\cup C$ receives a total charge of at least $3$, then
\begin{align*}
|E(G)|\ge& 2|D|+|A_3||D\setminus D_3|+3|B|+3|C|\\
\ge& 3kn -3|A|-|D_3|\\
=& 3kn-6-3|A_3|-|D_3|\\
\ge& 3kn-6-3n+2.
\end{align*}
So we suppose that there exists $b\in B_i$ with total charge at most $2.5$. If $N(b)\cap A_3=\emptyset$, then $C_1\subset N(b)$, and $b$ has charge at least $(n+1)/2$, so this is not the case. Thus, $|N(b)\cap A_3|=1$. Let $N(b)\cap A_3=\{a\}$, and let $a'\in A\setminus a$. Let $c\in N(b)\cap C_j$ be $b$'s only neighbor in $C$. Such a vertex must exist as there is a path of length $2$ from $b$ to $a'$. As $b$ has charge only $2.5$, $N(c)\cap A_3=\{a'\}$. Note that this argument also implies that $A_3=\{a,a'\}$. 

We complete this case in a manner similar to when $|A_3|=1$. Let $u\in B_i$ (the case for $w\in C_i$ is symmetric). 
If $N(u)\cap A_3=\emptyset$, then $C_1\subseteq N(u)$, so $u$ has charge at least $(n+1)/2$.
If $N(u)\cap A_3=A_3$,
then $u$ has charge at least $3$ (in fact, exactly $3$). 
If $N(u)\cap A_3= \{a\}$, then $\emptyset\subsetneq C\setminus (N(a) \cup C_i)\subset N(u)$, so $u$ has charge at least $2+|C\setminus (N(a) \cup C_i)|/2$. 
The only way that $u$ has charge less than $3$ in this case is if there exists $w\in C_j$ with $j\ne i$, such that $C\setminus (N(a) \cup C_i)=\{w\}$.
Note that in this case $N(w)\cap A_3=\{a'\}$. The charge of $w$ is at least $2.5$, so the combined charge of $u$ and $w$ is at least $5$.
Now let $U$ be the set of vertices $u'\in B\setminus B_j$ with weight $2.5$ satisfying $N(u')\cap A_3= \{a\}$.
Thus $u'w\in E(G)$, and the combined charge of $U$ and $w$ is at least $3|U|+2$.

Now consider $U'$, the set of vertices $u''\in B_j$ with charge $2.5$ satisfying $N(u'')\cap A_3= \{a\}$.
Thus $N(U')/cap C=\{w'\}$, and as $N(U)\cap C=\{w\}$, it follows that $w\in C_i$.
Further, the total charge of $U'$ and $w'$ is at least $3|U'|+2$.
Very similar conclusions hold for the case of $N(u)\cap A_3= \{a'\}$.
In conclusion, the total charge of $B\cup C$ is at least $3|B\cup C|-4$. Thus,
\begin{align*}
|E(G)|\ge& 2|D|+|A_3||D\setminus D_3|+3|B|+3|C|-4\\
\ge& 3kn -3|A|-|D_3|-4\\
=& 3kn-|D_3|-16\\
\ge& 3kn-n-14\\
>&3kn-3n.
\end{align*}
\end{proof}

When $k=3$, it is much easier to determine $\sat(K_3,K_k^n)$ for all values of $n$.

\begin{theorem}\label{k3tripartite}
$\sat(K_3,K_k^n) = 6n-6$.
\end{theorem}

\begin{proof}
Observe that $n-1+5/n>3$ for all $n\ge 2$.
Thus $\sat(K_3,K_k^n) \le 6n-6$ by Lemma~\ref{K3ub}.
Through the remainder of the proof we perform all arithmetic modulo $3$.

Let $G$ be a $K_3$-saturated subgraph of $K_3^n$.
Let $\delta_i$ denote the minimum degree in $G$ among the vertices in $V_i$.
Assume that $\delta_1\le \delta_2\le \delta_3$.
Each vertex in $V_i$ either has a neighbor in both $V_{i+1}$ and $V_{i+2}$ or is completely joined to $V_{i+1}$ or $V_{i+2}$; thus $\delta(G)\ge 2$.

Let $v_i^1$ be a vertex in $V_i$ with degree $\delta_i$.
Every vertex in $V_{i+1}\cup V_{i+2}$ that is not adjacent to $v_i^1$ has at least one neighbor among the $\delta_i$ neighbors of $v_i^1$.
Thus there are at least $2n-\delta_i$ edges joining $V_{i+1}$ and $V_{i+2}$.
Furthermore, there are at least $\delta_i n$ edges incident to the vertices in $V_i$.
If $\delta_i \ge 4$, then $E(G)\ge 4n+2n-4 = 6n-4$.
Thus we may assume that $\delta_i\le 3$ for all $i\in\{1,2,3\}$.

If $\delta_1 = \delta_2 = \delta_3 = 2$, then there are at least $2n-2$ edges joining each pair of $V_1$, $V_2$, and $V_3$.
Thus $|E(G)|\ge 6n-6$.

Now suppose that $\delta_1 = 2$ and $\delta_3 = 3$.
Every vertex of degree $2$ in $V_1$ is adjacent to a vertex of degree at least $n$ in $V_3$.
Therefore, there are at least $2n-3$ edges joining $V_1$ and $V_2$, and $V_3$ has degree sum at least $3(n-1)+n$.
Thus $|E(G)|\ge 2n-3+3(n-1)+n = 6n-6$.

Finally assume that $\delta_1 = \delta_2 = \delta_3 = 3$.
A vertex of degree $3$ in $V_i$ has a neighbor that is incident to $n-2$ edges joining $V_{i+1}$ and $V_{i+2}$.
Thus for each $j,l\in\{1,2,3\}$, $j\neq l$, there is a vertex $x_{j,l}$ that is incident to $n-2$ edges joining $V_j$ and $V_l$.
If $x_{1,2}$, $x_{1,3}$ and $x_{2,3}$ are distinct, then $G$ contains three vertices of degree at least $n-1$.
It follows that $|E(G)|\ge \frac{1}{2}(3(n-1) + 3(3n-3)) = 6n-6$.
If, without loss of generality, $x_{1,2} = x_{1,3}$, then $d(x_{1,2})\ge 2n-4$ and $d(x_{2,3})\ge n-1$.
Thus $|E(G)|\ge \frac{1}{2}(3n-5 + 3(3n-2)) = \frac{1}{2}(12n-11)>6n-6$.
\end{proof}

\section{$K_t$-saturated subgraphs for $t\ge 4$}\label{constructions}

In this section we provide constructions of $K_t$-saturated subgraphs of $K_k^n$ of small size for $t\ge 4$.
We start with natural generalizations of Constructions~\ref{exk>n} and \ref{exn>k}.

\begin{construction}\label{exnk}
Let $k\ge 2t-4$, and let $S=\{v_1^1,\ldots,v_{2t-4}^1\}$.
To construct $G_{k,n,t}$, place a complete graph on $S$ and remove the $t-2$-edge matching $\{v_1^1v_2^1,v_3^1v_4^1,\ldots,v_{2t-5}^1v_{2t-4}^1\}$.
Now, for $r\in\{1,3,\ldots,2t-5\}$ completely join $V_r-v_r^1$ and $V_{r+1}-v_{r+1}^1$.
Finally, add all edges from $K_k^n$ joining $S$ and $\overline{S}$.
That is,
\begin{align*}
	E(G_{k,n,t}) = 	&\bigg[ \{v_r^1v_s^1: r\le 2t-4, s\le 2t-4, r\neq s\}\setminus\{v_1^1v_2^1,v_3^1v_4^1,\ldots,v_{2t-5}^1v_{2t-4}^1\}\bigg] \\
									&\cup\{v_r^iv_{r+1}^j: i\ge 2, j\ge 2, r\in\{1,3,\ldots,2t-5\}\}\\
									&\cup\{v_r^iv_s^1: i\ge 2, r\le 2t-4, s\le 2t-4, r\neq s\}\\
									&\cup\{v_r^iv_s^1: i\le n, r> 2t-4, s\le 2t-4\}.
\end{align*}
The number of edges in $G_{k,n,t}$ is
\begin{align*}
	|E(G_{k,n,t})| 	&=\binom{2t-4}{2}-(t-2)+(t-2)(n-1)^2\\
									&\hspace{1.7in}+(2t-4)(2t-5)(n-1)+(2t-4)(k-2t+4)n\\
									&=(t-2)n^2+(2t-4)kn-2(2t-4)n-\binom{2t-4}{2}.
\end{align*}
\end{construction}

\begin{construction}\label{exn>>k}
Let $k\ge 2t-3$, and let $S=\{v_1^1,\ldots,v_1^{2t-3}\}$.
To construct $H_{k,n,t}$, begin by placing a complete graph on $S$ and removing the $2t-3$-cycle with edges $\{v_1^rv_1^s: |r-s|\in\{t-2,t-1\}\}$.
Finally, add all edges from $K_k^n$ joining $S$ and $\overline{S}$.
That is, 
\begin{align*}
	E(H_{k,n,t}) = 	&\bigg[\{v_1^rv_1^s: r\le 2t-3, s\le 2t-3\, r\neq s\}\setminus\{v_1^rv_1^s: |r-s|																		\in\{t-2,t-1\}\}\bigg]\\
								&\cup\{v_i^rv_1^s: i\ge 2, r\le 2t-3, s\le 2t-3, r\neq s\}\\
								&\cup\{v_i^rv_1^s: i\in[n], r> 2t-3, s\le 2t-3\}.
\end{align*}
The number of edges in $H_{k,n,t}$ is
\begin{align*}
|E(H_{k,n,t})|	&= (2t-3)(t-3) + (2t-3)(2t-4)(n-1)+(k-2t+3)(2t-3)n\\
							&= (2t-3)kn-(2t-3)n-(2t-3)(t-1).
							\end{align*}
\end{construction}

It is tedious but straightforward to verify that both $G_{k,n,t}$ and $H_{k,n,t}$ are $K_t$-saturated subgraphs of $K_k^n$ for $k\ge 2t-4$ and $k\ge 2t-3$, respectively.  Consequently, we have the following bound on $\sat(K_k^n,K_t)$ for $t\ge 3$ and $k\ge 2t-3$.

\begin{theorem}\label{generalbnd}
If $t\ge 3$ and $k\ge 2t-3$, then
\begin{eqnarray*}
\sat(K_k^n,K_t)\le \min\left\{
\begin{array}{l}
(2t-4)kn+(t-2)n^2-2(2t-4)n-\binom{2t-4}{2},\\
(2t-3)kn-(2t-3)n-(2t-3)(t-1)
\end{array}
\right\}.
\end{eqnarray*}
\end{theorem}

As $G_{k,n,t}$ and $H_{k,n,t}$ are structurally similar to the unique minimal saturated graphs from Theorem \ref{K3main}, we conjecture that the bound in Theorem~\ref{generalbnd} is sharp when $k$ is sufficiently large relative to $t$ and $n\ge 2$.

The remaining constructions in this section follow the same general approach to building a $K_t$-saturated subgraph of $K_k^n$.  First we select a small set of vertices $S$ and construct on $S$ a $K_t$-free graph that, for each choice of a two partite sets in $K_k^n$, contains a copy of $K_{t-2}$ on $t-2$ vertices not lying in the two selected partite sets.  We then add all edges in $K_k^n$ joining $S$ and $\overline{S}$.  Finally, if necessary, iteratively add edges joining vertices in $S$ provided that these edges do not complete any $t$-cliques.  The resulting graph is a $K_t$-saturated subgraph of $K_k^n$ and the number of edges is on the order of $|S|nk$.

We now turn our attention to are $K_t$-saturated subgraphs of $K_k^n$ for $k\in\{t,\ldots,2t-5\}$ where it seems that there may be a rich structure to the family of minimal $K_t$-saturated subgraphs of $K_k^n$.  We present two additional constructions. The first applies to all values of $t$, $k$, and $n$, while the second applies only when $t$ is even and $k\ge \frac{3}{2}(t-2)$.  

\begin{construction}\label{lexo}
Let $k\ge t$ and construct the graph $F_{k,n,t}$ as follows.  First, list all $t-2$-element subsets of $[t]$ in lexicographic order.  Thus for any $R\in\binom{[t]}{t-2}\setminus\{\{1,\ldots,t-2\}\}$, there is a $t-2$-set $R'$ preceding $R$ that contains the $t-3$ lowest elements of $R$.
Begin by letting $S$ contain one vertex from each of $V_1\ldots,V_{t-2}$ and constructing a $t-2$-clique on those vertices.
For each subsequent set $R$ in the ordering of $\binom{[t]}{t-2}$, add a vertex from $V_{\max(R)}$ to $S$ and join it to a $t-3$-clique in $S$ whose vertices lie in the sets indexed by $R-\max(R)$.
Thus for each set $R\in\binom{[t]}{t-2}$ there is a $t-2$-clique whose vertices lie in the partite sets indexed by $R$.
Next, add all edges from $K_k^n$ joining $S$ and $\overline{S}$.
Finally, iteratively add edges from $K_k^n$ joining vertices in $S$ provided that those edges do not complete any $t$-cliques.
\end{construction}

\begin{construction}\label{tri}
For $t=2m$ with $m\ge 3$, and $k\ge \frac{3}{2}(t-2)$ we construct the graph $I_{k,n,t}$ as follows.
Let $$S=\{v_1^1,v_1^2,v_2^1,v_2^2,\ldots,v_{3(t-2)/2}^1,v_{3(t-2)/2}^2\}$$
and start with the induced subgraph of $K_k^n$ on $S$. 
If $t\equiv2\pmod4$  (see Figure~\ref{tf1}), then for $i\in\{0,\ldots,\frac{t-2}{4}-1\}$, delete the edges of the following triangles:
$$\{v_{6i+1}^1,v_{6i+2}^1,v_{6i+3}^1\},\{v_{6i+4}^1,v_{6i+5}^1,v_{6i+6}^1\},\{v_{6i+1}^2,v_{6i+3}^2,v_{6i+5}^2\},\{v_{6i+2}^2,v_{6i+4}^2,v_{6i+6}^2\}.$$
\begin{figure}
\begin{center}
\includegraphics[scale=.4]{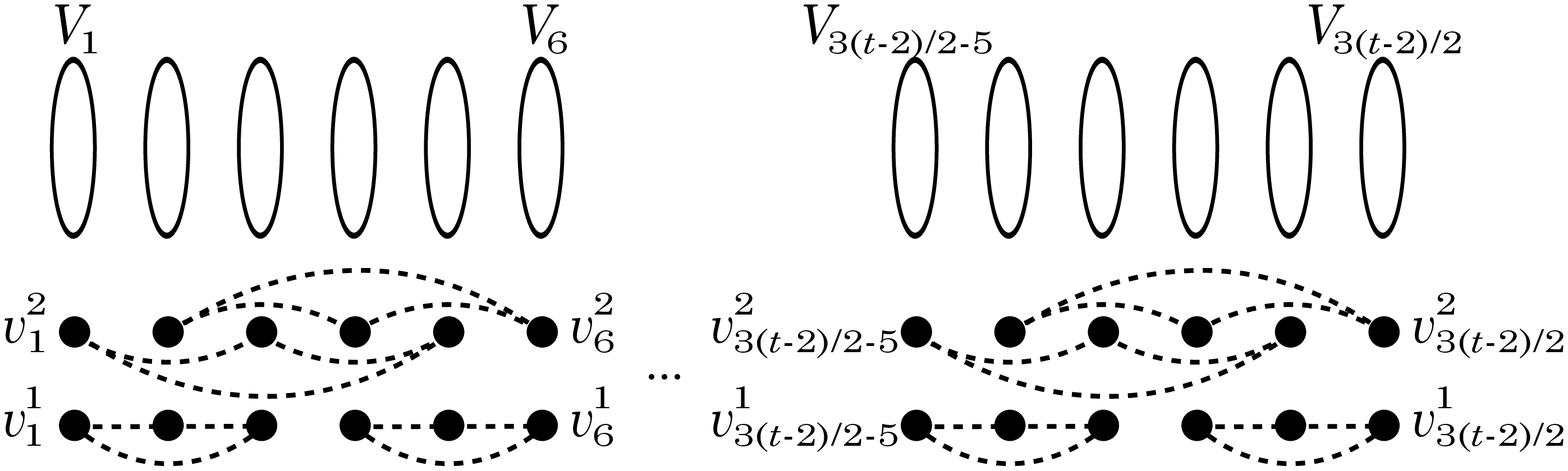}
\caption{Constructing $I_{k,n,t}$: Nonedges in $S$ when $k\equiv2\pmod4$.}\label{tf1}
\end{center}
\end{figure}
If $t\equiv0\pmod4$ (see Figure~\ref{tf2}), then for $i\in \{0,\ldots,\frac{t-4}{4}-1\}$, delete the edges of the triangles 
$$\{v_{6i+1}^1,v_{6i+2}^1,v_{6i+3}^1\},\{v_{6i+4}^1,v_{6i+5}^1,v_{6i+6}^1\},\{v_{6i+1}^2,v_{6i+3}^2,v_{6i+5}^2\},\{v_{6i+2}^2,v_{6i+4}^2,v_{6i+6}^2\}$$
and also delete the edges of the triangles
\begin{align*}
\{v_{\frac{3}{2}(t-2)-8}^1,v_{\frac{3}{2}(t-2)-7}^1,v_{\frac{3}{2}(t-2)-6}^1\},
\{v_{\frac{3}{2}(t-2)-5}^1,v_{\frac{3}{2}(t-2)-4}^1,v_{\frac{3}{2}(t-2)-3}^1\},
\{v_{\frac{3}{2}(t-2)-2}^1,v_{\frac{3}{2}(t-2)-1}^1,v_{\frac{3}{2}(t-2)}^1\},\\
\{v_{\frac{3}{2}(t-2)-8}^2,v_{\frac{3}{2}(t-2)-5}^2,v_{\frac{3}{2}(t-2)-2}^2\},
\{v_{\frac{3}{2}(t-2)-7}^2,v_{\frac{3}{2}(t-2)-4}^2,v_{\frac{3}{2}(t-2)-1}^2\},
\{v_{\frac{3}{2}(t-2)-6}^2,v_{\frac{3}{2}(t-2)-3}^2,v_{\frac{3}{2}(t-2)}^2\}.
\end{align*}
\begin{figure}
\begin{center}
\includegraphics[scale=.4]{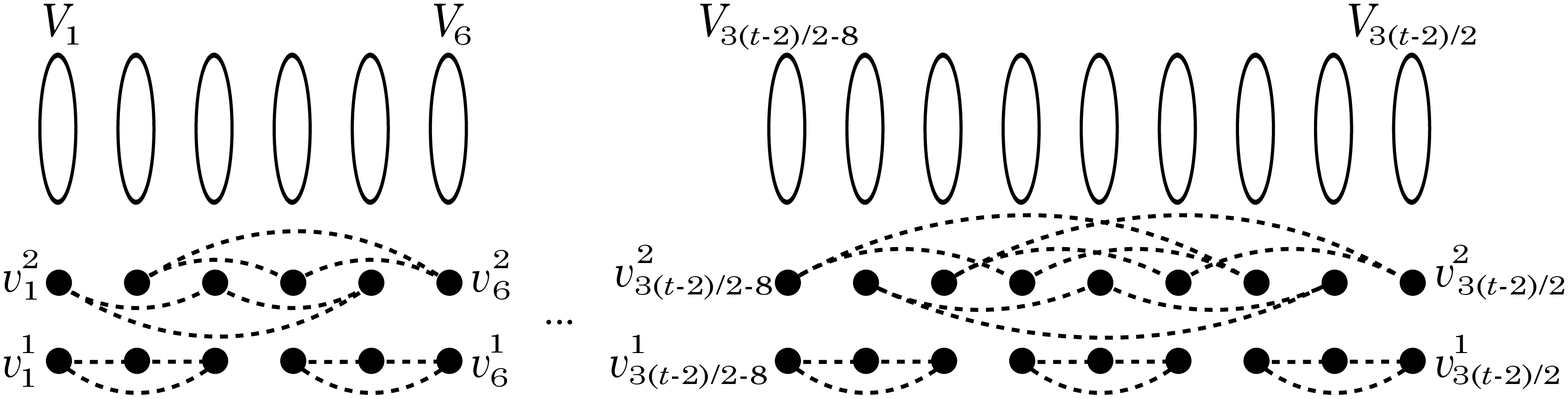}
\caption{Constructing $I_{k,n,t}$: Nonedges in $S$ when $k\equiv0\pmod4$.}\label{tf2}
\end{center}
\end{figure}
To complete the construction, add all edges in $K_k^n$ joining vertices in $S$ to vertices in $\overline{S}$.
\end{construction}

Recall that both $F_{k,n,t}$ and $I_{k,n,t}$ have size on the order of $|S|nk$, for those sets $S$ given in their respective constructions.  This yields the following theorem, the details of which are again tedious but straightforward, and hence left to the reader.  

\begin{theorem}\label{thmtri}
\begin{enumerate}
\item For $k\ge t\ge 4$ and $n\ge 2$, $F_{k,n,t}$ is a $K_t$-saturated subgraph of $K_k^n$, so
$$\sat(K_k^n,K_t)\le 3(t-2)nk+o(nk).$$
\item For even $t\ge 6$ and $k\ge 3(t-2)$, $I_{k,n,t}$ is a $K_t$-saturated subgraph of $K_k^n$, so
$$\sat(K_k^n,K_t)\le \frac{1}{2}(t^2+t-6)nk+o(nk).$$
\end{enumerate}
\end{theorem}

We end with the following question, motivated by the differing number of edges in $F_{k,n,t}$ and $I_{k,n,t}$.

\begin{question}
Is there a linear function $f(t)$ such that for all $k\ge t\ge 3$ and $n$ sufficiently large, $\sat(K_k^n,K_t)\le f(t)kn$?
\end{question}

\end{document}